\documentclass{amsart}
\usepackage{amssymb,euscript,amsmath, mathrsfs, amscd}
\usepackage[dvips]{graphicx}
\usepackage[dvips]{color}

\newcounter{ENUM}
\newcommand{\itm}{\item}

\newenvironment{Ilist}{\renewcommand{\theENUM}{\Roman{ENUM}}\renewcommand{\itm}{\addtocounter{ENUM}{1}\item[(\theENUM)]}\begin{itemize}\setcounter{ENUM}{0}}{\end{itemize}}

\newcommand{\margh}[1]{}

\def\risom{\overset{\sim}{\rightarrow}}

\input xy
\xyoption{all}
\CompileMatrices

\def\ZZ{{\mathbb Z}}

\def\cE{{\mathcal E}}

\def\cG{{\mathcal G}}

\def\cS{{\mathcal S}}

\def\cC{{\mathcal C}}

\def\cU{{\mathcal U}}

\def\cX{{\mathcal X}}

\def\sE{{\mathscr E}}
\def\sF{{\mathscr F}}
\def\sG{{\mathscr G}}
\def\sI{{\mathscr I}}

\def\sL{{\mathscr L}}
\def\sM{{\mathscr M}}
\def\sO{{\mathscr O}}

\def\sHom{{\mathscr H}om}

\def\fg{{\mathfrak g}}

\def\vp{\varphi}

\def\dv{\operatorname{div}}

\def\Hom{\operatorname{Hom}}

\def\GL{\operatorname{GL}}

\def\Spec{\operatorname{Spec}}
\def\Pic{\operatorname{Pic}}

\def\res{\operatorname{res}}

\def\spn{\operatorname{span}}

\def\st{\operatorname{st}}

\def\gg{\operatorname{gg}}
\def\dg{\operatorname{dg}}
\def\ns{\operatorname{ns}}
\def\lcm{\operatorname{lcm}}

\newtheorem{thm}{Theorem}[section]
\newtheorem{prop}[thm]{Proposition}
\newtheorem{lem}[thm]{Lemma}
\newtheorem{cor}[thm]{Corollary}

\theoremstyle{definition}
\newtheorem{defn}[thm]{Definition}

\newtheorem{ex}[thm]{Example}

\theoremstyle{remark}

\newtheorem{rem}[thm]{Remark}
\newtheorem{rems}[thm]{Remarks}

\numberwithin{equation}{section}
\numberwithin{figure}{section}

\begin{document}
\title{Special determinants in higher-rank Brill-Noether theory}
\author{Brian Osserman}
\thanks{This work was partially carried out while the author was visiting
the Isaac Newton Institute during the 2011 program on Moduli Spaces, and
was partially supported by NSA grant H98230-11-1-0159.}

\begin{abstract} Continuing our previous study of modified expected 
dimensions for rank-$2$ Brill-Noether loci with prescribed special
determinants, we introduce a general framework which applies \textit{a 
priori} for arbitrary rank, and use it to prove modified expected dimension
bounds in several new cases, applying both to rank $2$ and to higher rank.
The main tool is the introduction of generalized alternating Grassmannians,
which are the loci inside Grassmannians corresponding to subspaces which 
are simultaneously isotropic for a family of multilinear alternating forms
on the ambient vector space. In the case of rank $2$ with $2$-dimensional 
spaces of sections, we adapt arguments due to Teixidor i Bigas to show that
our new modified expected dimensions are in fact sharp.
\end{abstract}

\maketitle

\section{Introduction}

The purpose of the present paper is to continue the systematic study
of higher-rank Brill-Noether loci with fixed special determinant initiated
in \cite{os16}. Given a smooth projective curve $C$ of genus $g$, and a
line bundle $\sL$ on $C$, we set up a general framework for proving 
dimension lower bounds for Brill-Noether loci with fixed determinant
$\sL$, expressed in terms of $h^1(C,\sL)$. Although our immediate goal is
a sharp understanding of the rank-$2$ case, the setup is carried out 
in full generality, including in arbitrary rank. We then apply it to
obtain concrete results in several families of cases for which the 
dimension of the space of sections considered is relatively small compared
to the rank.

Given $k,r$, denote by $\cG^k_{r,\sL}(C)$ the moduli stack of vector
bundles on $C$ of rank $r$ and fixed determinant $\sL$ together with a
$k$-dimensional space of global sections. The naive expected dimension
for $\cG^k_{r,\sL}(C)$ is
$$\rho-g:=(r^2-1)(g-1)-k(k-d+r(g-1)).$$

Our first theorem is as follows.

\begin{thm}\label{thm:main-1} Let $C$ be a smooth, projective curve of
genus $g$. Suppose $\sL \in \Pic^d(C)$, and $h^1(C,\sL)\geq m$. Given 
$r \geq 2$, let $\sE$ be a vector bundle of rank $r$ on $C$ with determinant 
$\sL$, and $V \subseteq H^0(C,\sE)$ a $k$-dimensional space of global 
sections. Suppose that in addition, one of the following conditions is
satisfied.
\begin{Ilist}
\itm $k=r$, and $V$ is not contained in any subbundle of $\sE$ of 
rank $r-2$.
\itm $k=r+1$, $m=1$, and no $r$-dimensional subspace of $V$ is contained 
in any subbundle of $\sE$ of rank $r-2$.
\itm $r=3$, $k=5 \text{ or }6$, $m=1$, and no $2$-dimensional subspace of 
$V$ is contained in any subbundle of $\sE$ of rank $1$.
\end{Ilist}

Then every component of $\cG^k_{r,\sL}(C)$ passing through the point 
corresponding to $(\sE,V)$ has dimension at least 
\begin{equation}\label{eq:mod-rho} \rho-g+m\binom{k}{r}.\end{equation}
\end{thm}

\begin{rems} 
(i) Even though we are only proving a lower bound on dimension,
a nondegeneracy hypothesis is required. This has not been the case
for previous work on the subject, but is expected to be a feature
of any further generalizations. The nondegeneracy hypothesis in cases (I)
and (II) is essentially a generic version of what Mukai calls 
``semiirreducibility'' in \cite{mu4}.

(ii) Only case (I) gives new results for rank $2$, since the main results of
\cite{os16} proved in particular the same dimension bound as above
in the cases $r=2,m \leq 2$, and with $k$ arbitrary.

(iii) Although case (III) may appear special, recall that $d$ and $g$ are
allowed to vary, so they still contain infinite families of rank-$3$
Brill-Noether loci, including some particularly interesting examples;
see Example \ref{ex:3-6} below.

(iv) Our arguments also work for families of special determinants, and can
thus be used to study the variable determinant case as well. See Theorem
\ref{thm:main-families} below for a precise statement.
\end{rems}

Under somewhat stronger nondegeneracy hypotheses than those imposed
in Theorem \ref{thm:main-1}, the dimension statements for cases (I) and
(II) may be approached via direct analysis. This has been done in the
literature for the case of varying determinant as follows: for 
$r=k=2$ by Teixidor i Bigas in \cite{te4}, and for the more general
cases by Bradlow-Garc\'ia-Prada-Mu\~noz-Newstead \cite{b-g-m-n1},
by Bradlow-Garc\'ia-Prada-Mercat-Mu\~noz-Newstead \cite{b-g-m-m-n1},
and by Bhosle-Brambila-Paz-Newstead \cite{b-b-n1}. However, the same
constructions may be applied to the fixed determinant case; see
Grzegorczyk and Newstead \cite{g-n1}. Of particular
note is that these constructions also show that the dimension lower bounds 
of Theorem \ref{thm:main-1} are sharp (still under the stronger 
nondegeneracy hypotheses). 

We illustrate these methods by following and elaborating on the arguments of 
Teixidor to verify that case (I) of Theorem \ref{thm:main-1} is sharp for 
$r=2$, even without additional nondegeneracy hypotheses.
To state the theorem, we denote by $\cG^{2,\gg}_{2,\sL}(C)$ the open substack
of $\cG^{2}_{2,\sL}(C)$ on which the bundle is generically generated by
the chosen space of global sections, and by
$\cG^{2,\st}_{2,\sL}(C)$ the open substack on which the underlying bundle is
stable.

\begin{thm}\label{thm:main-2} Let $C$ be a smooth, projective curve of
genus $g$. Suppose $\sL \in \Pic^d(C)$, and $h^1(C,\sL)=m$. 

Then $\cG^{2,\gg}_{2,\sL}(C)$ is nonempty if and only if $h^0(C,\sL)>0$, and in 
this case is irreducible of dimension $\rho-g+m$.

If $C$ is Brill-Noether general with respect to $\fg^1_{d'}$'s for 
all $d'$, then $\cG^2_{2,\sL}(C)$ has dimension $\rho-g+m$. If further
$h^0(C,\sL)>0$, then $\cG^{2,\gg}_{2,\sL}(C)$ is dense in
$\cG^2_{2,\sL}(C)$, and in particular $\cG^2_{2,\sL}(C)$ is irreducible.

Finally, when $h^0(C,\sL)>0$, the stack $\cG^{2,\st}_{2,\sL}(C)$ contains a 
nonempty open substack of $\cG^{2,\gg}_{2,\sL}(C)$ if $C$ is nonhyperelliptic 
and $d\geq 3$ or if $C$ is hyperelliptic and $d \geq 5$.
\end{thm}

We conclude with a discussion of the prospects for further generalization,
and speculation on the possible form of sharp dimension bounds in rank $2$.
In the process, we investigate several examples from the literature, and
find that their constructions of Brill-Noether loci having greater than the
expected dimension can be explained by our results.

As in \cite{os16}, the techniques underlying Theorem \ref{thm:main-1}
(and the more general framework) involve suitable generalizations of
symplectic Grassmannians. Beyond introducing families of alternating
forms as was already considered in \cite{os16}, to treat the higher-rank
case we consider multilinear forms instead of just bilinear forms. This
adds additional complications, but due to some simplifications in the
overall strategy we are able to prove Theorem \ref{thm:main-1}. There is
great potential for further generalization, but it will involve a more
delicate analysis of how to translate the (multi)linear algebra into
suitable nondegeneracy conditions.

In contrast, as in \cite{te4}, Theorem \ref{thm:main-2} is proved using a 
careful study of extensions, and the proof is not expected to generalize. 
Systematic use of stack-theoretic dimension counting simplifies the 
arguments. 

Others have previously considered the two directions of generalization of 
symplectic Grassmannians discussed above. Subspaces simultaneously 
isotropic for families of alternating forms have been studied by Buhler,
Gupta and Harris \cite{b-g-h1} in the context of group theory, 
while Tevelev \cite{te6} has studied subspaces isotropic for generic 
multilinear alternating forms. However, in both cases the
focus was on nonemptiness questions, whereas in our case we need to
develop criteria for the spaces to be smooth of expected dimension at 
a particular point.

\section*{Acknowledgements} I would like to thank Peter Newstead for 
helpful conversations.

\section{Generalized alternating Grassmannians}

Let $X$ be a scheme, and $\sE$ a vector bundle on $X$ of rank $n$.
Recall that an $r$-linear alternating form on $\sE$ is a morphism 
$$\langle,\dots,\rangle:\bigwedge^r \sE \to \sO_X.$$
A subbundle $\sF \subseteq \sE$ is isotropic for 
$\langle,\dots,\rangle$ if the restriction of $\langle,\dots,\rangle$ to 
$\bigwedge^r \sF$ is equal to $0$. The subbundle $\sF$ is degenerate for
$\langle,\dots,\rangle$ if the induced morphism
$\bigwedge^{r-1} \sF \to \sE^*$ is equal to $0$.

Suppose we are given a collection
$$\underline{\langle,\dots,\rangle}
=\{\langle,\dots,\rangle_1, \dots, \langle,\dots, \rangle_m\}$$ 
of $m$ $r$-linear alternating forms on $\sE$.
Then we make the following definition:

\begin{defn} Given $k<n$, we have the {\bf generalized alternating 
Grassmannian} 
$GAG(k,\sE,\underline{\langle,\dots,\rangle})$ defined as the closed
subscheme of $G(k,\sE)$ whose points correspond to subbundles which are 
simultaneously isotropic for every 
$\langle,\dots,\rangle_i \in \underline{\langle,\dots,\rangle}$.
\end{defn}

If $X$ is a point and the forms are
sufficiently general, the generalized alternating Grassmannian has
codimension $m\binom{k}{r}$ in $G(k,\sE)$. However, the case of interest
for us is not completely general, so we have to carry out a closer analysis.
The case $r=2, m\leq 2$ was handled in \cite{os16}.
We will see that the same criterion considered in \textit{loc.\ cit.} (which 
does not hold in general) also holds when $k=r$, or when $m=1$ and $k=r+1$. 
We first give a general description
translating smoothness into (multi)linear algebra.

\begin{lem}\label{lem:tangent-description} Suppose $\sE$ is a vector bundle
of rank $n$ on a scheme $X$, and $\langle,\dots,\rangle_i$
for $i=1,\dots,m$ are $r$-linear alternating forms on $\sE$.
Given a field $K$, and a $K$-valued point $x$ of $X$, suppose we have 
$V \subseteq \sE|_x$ corresponding to a $K$-valued point 
$z$ of $GAG(k,\sE|_x,\underline{\langle,\dots,\rangle})$. Then at the (image
of the) point $z$, we have 
$GAG(k,\sE,\underline{\langle,\dots,\rangle})$ 
smooth over $X$ of codimension $m\binom{k}{r}$ inside $G(k,\sE)$
if and only if the induced map of $K$-vector spaces
\begin{equation}\label{eq:kermap} 
(\bigwedge^r V)^{\oplus m} \to \Hom(V,\sE|_x/V)^*
\end{equation}
is injective, where the map is determined by 
$$\sum_{i=1} ^m v_{1,i} \wedge \dots \wedge v_{r,i} \mapsto (\vp \mapsto
\sum_{i=1}^m \sum_{j=1}^r 
\langle v_{1,i},\dots,\varphi(v_{j,i}),\dots,v_{r,i}\rangle_i).$$
\end{lem}

The following lemma is standard, but we state it for convenience of
notation:

\begin{lem}\label{lem:lci-smooth} Let $X \to S$ be smooth of relative
dimension $d$, and $Z \subseteq X$
a closed subscheme. Suppose that for some $z \in Z$, with image $s \in S$,
we have that the ideal sheaf $\sI_Z$ is generated by $c$ elements locally
near $z$, and that the fiber $Z_s$ is smooth at $z$ over $\Spec \kappa(s)$, of
codimension $c$ in $X_s$. Then $Z$ is smooth at $z$ of relative dimension 
$d-c$ over $S$.
\end{lem}

\begin{proof} This follows essentially immediately from Proposition 2.2.7 of
\cite{b-l-r}. Indeed, if $f_1,\dots,f_c$ are local generators for $\sI_Z$
near $z$, then applying part (c) of {\it loc.\ cit.} to the fibers $X_s$ and 
$Z_s$ we find that the differentials $df_1,\dots,df_c$ must be linearly 
independent in $\Omega^1_{X/S}|_z$. But then applying part (d) of
{\it loc.\ cit.} to $X$ and $Z$, we find that $Z$ is smooth at $z$ of 
relative dimension $d-c$, as desired.
\end{proof}

\begin{proof}[Proof of Lemma \ref{lem:tangent-description}]
Recall that if $E$ is a $K$-vector space, and
$V \subseteq E$ corresponds to a $K$-valued point $z$ of the classical
Grassmannian $G(k,E)$, then the tangent 
space to $G(k,E)$ at $z$ is given by $\Hom(V,E/V)$. Now, if 
$\langle ,\dots,\rangle$ is an $r$-linear alternating form on $E$, and
$V$ is isotropic for $\langle,\dots,\rangle$, then every tangent vector
of $G(k,E)$ at $z$ gives us an $r$-linear alternating form 
$\langle,\dots,\rangle^{\vp}$ as follows: 
if the tangent vector is given by $\vp \in \Hom(V,E/V)$, the associated
form is determined by sending $v_1 \wedge \dots \wedge v_r \in \bigwedge^r V$
to 
$$\sum_{i=1}^r \langle v_1,\dots,\vp(v_i),\dots,v_r\rangle.$$
This gives us a map 
$$\Hom(V,E/V) \to \left(\bigwedge^r V\right)^*.$$
Thus, the given alternating forms induce a map
$$\Hom(V,E/V) \to \bigoplus_{i=1}^m \left(\bigwedge^r V\right)^*.$$
It is easy to see that the tangent space to
$GAG(k,E,\underline{\langle,\dots,\rangle})$ is precisely the kernel of 
this map. Note also that this map is dual to \eqref{eq:kermap} (with $E$
in place of $\sE|_x$). Now, we know 
that $GAG(k,E,\underline{\langle,\dots,\rangle})$ is locally cut out by
$m\binom{k}{r}$ equations inside $G(k,E)$, so every component of 
$GAG(k,E,\underline{\langle,\dots,\rangle})$ has codimension
at most $m \binom{k}{r}$ in $G(k,E)$, and 
$GAG(k,E,\underline{\langle,\dots,\rangle})$ is smooth at $z$ of pure
dimension $k(n-k)-m\binom{k}{r}$ if and only if the tangent space at $z$
has dimension is $k(n-k)-m\binom{k}{r}$,
if and only if the above map is surjective. This
in turn is equivalent to the injectivity of \eqref{eq:kermap}, again with 
$E$ in place of $\sE|_x$.

Considering the situation of the lemma statement, if we set $E=\sE|_x$,
recalling that smoothness of a fiber may be checked after extending the 
base field we conclude from the above that the fiber over $x$ of 
$GAG(k,\sE,\underline{\langle,\dots,\rangle})$ is smooth of codimension
$m\binom{k}{r}$ in $G(k,\sE)$ at the point $z$ if and only if \eqref{eq:kermap}
is injective. Finally, we conclude the statement of the lemma by
applying Lemma \ref{lem:lci-smooth}.
\end{proof}

We thus conclude the following general statement on loci of subbundles
contained in two given subbundles.

\begin{prop}\label{prop:iso-intersect} Suppose $\sE$ is a vector bundle
of rank $n$ on an algebraic stack $\cX$ of finite type over a universally
catenary scheme $S$, and $\langle,\dots,\rangle_i$
for $i=1,\dots,m$ are $r$-linear alternating forms on $\sE$. Let $\sF$ and
$\sG$ be subbundles of $\sE$ of ranks $s$ and $t$, both isotropic with
respect to all of the $\langle,\dots,\rangle_i$. Let
$\cG(k,\sF\cap\sG)$ denote the closed substack of $\cG(k,\sE)$ parametrizing
rank-$k$ subbundles of $\sE$ contained in both $\sF$ and $\sG$.
Suppose that for some field $K$ and some $K$-valued point $x$ of $\cX$, we 
have $V \subseteq \sF|_x \cap \sG|_x$
such that the map \eqref{eq:kermap} is injective.
Then every component of $\cG(k,\sF\cap\sG)$ passing
through the point corresponding to $V$ has codimension at
most
$$k(2n-s-t)-m\binom{k}{r}$$
in $\cG(k,\sE)$.
\end{prop}

\begin{proof}
We first reduce from the case of an algebraic stack $\cX$ to the case of
a scheme $X$ by letting $X \to \cX$ be a smooth cover, and pulling back
the bundles, as in the argument for Corollary 3.7 of \cite{os16}. Then
$X$ is of finite type over $S$, and hence universally catenary.

Now, we can realize $G(k,\sF \cap \sG)$ as follows: note that
$G(k,\sF)$ is smooth over $S$, and has pure codimension $k(n-s)$ everywhere
in $G(k,\sE)$. Because $\sF$ and $\sG$ are isotropic for the
$\langle ,\dots,\rangle_i$, the universal subbundle on $G(k,\sF)$, together
with the pullback of $\sG$, induce a morphism
$$G(k,\sF) \to
GAG(k,\sE,\underline{\langle,\dots,\rangle}) \times_S 
GAG(t,\sE,\underline{\langle,\dots,\rangle}).$$
Denote the latter product by $P$, and let $I \subseteq P$ be the closed
subscheme determined by the incidence correspondence. Then $G(k,\sF\cap \sG)$ 
is precisely the preimage of the incidence correspondence, so because $X$
is universally catenary it suffices to
show that $I$ is cut out locally at $x$ by $k(n-t)-m\binom{k}{r}$ equations
inside $P$. But we can construct $I$ as a relative Grassmannian of subbundles
of the universal bundle on the second factor 
$GAG(t,\sE,\underline{\langle,\dots,\rangle})$;
thus, $I$ is smooth over $GAG(t,\sE,\underline{\langle,\dots,\rangle})$ of relative
dimension $k(t-k)$. On the other hand, by hypothesis and Lemma 
\ref{lem:tangent-description} we have that $P$ is smooth over 
$GAG(t,\sE,\underline{\langle,\dots,\rangle})$ of relative dimension 
$k(n-k)-m\binom{k}{r}$ at the point $z$ corresponding to $V$.
Thus, by Proposition 2.2.7 of \cite{b-l-r}, we have that locally near 
$z$, the scheme $I$ is cut out by $k(n-t)-m\binom{k}{r}$ equations, as 
desired.
\end{proof}

We now consider in some special cases what it means for \eqref{eq:kermap}
to have a nontrivial kernel.
We observe that one way in which \eqref{eq:kermap} can fail to be
injective is if there is some $r$-dimensional subspace $W \subseteq V$
and some nonzero $K$-linear combination of the $\langle ,\dots,\rangle_i|_x$
for which $W$ is degenerate. In \cite{os16},
we saw that the converse holds when $r=2$ and $m=1,2$. However,
the converse does not hold in general. Nonetheless, we now observe that the
converse holds in two other situations, as follows.

\begin{prop}\label{prop:gag-crit} Let $E$ be a $K$-vector space,
$\underline{\langle,\dots,\rangle}$ an $m$-dimensional space of
$r$-linear alternating forms on $E$,
and $V \subseteq E$ an $k$-dimensional subspace.
Suppose either that $k=r$, or that $k=r+1$ and $m=1$. 
Then the map
\begin{equation}\label{eq:kermap2}
(\bigwedge^r V)^{\oplus m} \to \Hom(V,E/V)^*
\end{equation}
induced as in \eqref{eq:kermap} is injective if and only if
there is no nonzero 
$\langle,\dots,\rangle \in \underline{\langle,\dots,\rangle}$
which is degenerate on an $r$-dimensional subspace of $V$.
\end{prop}

\begin{proof} As remarked above, if some nonzero
$\langle,\dots,\rangle \in \underline{\langle,\dots,\rangle}$ is degenerate
on an $r$-dimensional subspace of $V$, then
\eqref{eq:kermap2} fails to be injective much more generally. 

Conversely, first suppose $k=r$, and let
$\langle,\dots,\rangle_i$ for $i=1,\dots,m$ be a basis for 
$\underline{\langle,\dots,\rangle}$ and
$v_1,\dots,v_r$ a basis for $V$. Then
$\bigwedge^r V$ is $1$-dimensional, with basis $v_1 \wedge \dots \wedge v_r$.
A nonzero element of the kernel of \eqref{eq:kermap2} may thus be written as
as $\sum_i c_i(v_1 \wedge \dots \wedge v_r)_i$, where the subscript $i$
denotes the $i$th place in the direct sum, and not all $c_i$ are $0$.
By definition, this means that for all $\vp \in \Hom(V,E/V)^*$, we
have 
$$\sum_{i=1}^m \sum_{j=1}^r c_i \langle v_1,\dots,\vp(v_j),\dots,v_r\rangle_i
=0.$$
Since we may choose $\vp(v_j)=0$ for all but one $j$, and $\vp(v_j)$
arbitrary for the remaining index, this implies that the span of any
$r-1$ of the $v_j$ is degenerate for
$\sum_{i=1}^m c_i \langle,\dots,\rangle_i$. We thus conclude that $V$
is likewise degenerate, proving the desired statement.

On the other hand, if $m=1$ and $k=r+1$, it is still true that every
nonzero element of $\bigwedge^r V$ is of the form 
$v_1 \wedge \dots \wedge v_r$ for some linearly independent $v_i \in V$,
so an element of the kernel of \eqref{eq:kermap2} is simply of the form
$v_1 \wedge \dots \wedge v_r$, and arguing as above we conclude that the
span of the $v_i$ is degenerate, as desired.
\end{proof}

The following proposition uses a variant approach to treat some additional
cases when $r=3$ and $m=1$, as in Case (III) of Theorem \ref{thm:main-1}.

\begin{prop}\label{prop:gag-crit-2} Let $E$ be a $K$-vector space,
$\langle,,\rangle$ a $3$-linear alternating form on $E$,
and $V \subseteq E$ a $k$-dimensional subspace, with $k \leq 6$.
Then the map
\begin{equation}\label{eq:kermap3}
\bigwedge^3 V \to \Hom(V,E/V)^*
\end{equation}
induced as in \eqref{eq:kermap} is injective if 
there is no $2$-dimensional subspace $W \subseteq V$ on which
$\langle,,\rangle$ is degenerate.
\end{prop}

\begin{proof} The significance of the restriction to $r=3$ and $k \leq 6$ is
that for any element of $\bigwedge^3 V$, there exists a basis $v_1,\dots,v_k$
of $V$ such that the given element may be expressed as 
$$v_1 \wedge v_2 \wedge v_3 + \text{(terms not involving $v_1$)}.$$
Indeed, this follows from the classification of $\GL(V)$ orbits of
$\bigwedge^3 V$ as described for instance in \S 1.4 and \S 2.2 of \cite{re1}.
Now, suppose that \eqref{eq:kermap3} is not injective, and choose a basis of
$V$ so that an element of the kernel has the above form. Then we can choose
$\vp \in \Hom(V,E/V)^*$ sending $v_i$ to $0$ for $i>0$, and $v_1$ to an
arbitrary element of $E$. By definition of \eqref{eq:kermap3}, we see that
$\langle v,v_2,v_3 \rangle = 0$ for all $v \in E$, or equivalently, 
that $\spn(v_2,v_3)$ is degenerate for $\langle,,\rangle$. We thus 
conclude the statement of the proposition.
\end{proof}

Putting together Propositions \ref{prop:iso-intersect}, \ref{prop:gag-crit},
and \ref{prop:gag-crit-2}, we immediately obtain the following corollary:

\begin{cor}\label{cor:iso-intersect} Suppose $\sE$ is a vector bundle
of rank $n$ on an algebraic stack $\cX$ of finite type over a
universally catenary scheme $S$,
and $\langle,\dots,\rangle_i$
for $i=1,\dots,m$ are alternating $r$-linear forms on $\sE$. Let $\sF$ and
$\sG$ be subbundles of $\sE$ of ranks $s$ and $t$, both isotropic with
respect to all of the $\langle,\dots,\rangle_i$. Let
$\cG(k,\sF\cap\sG)$ denote the closed substack of $\cG(k,\sE)$ parametrizing
rank-$k$ subbundles of $\sE$ contained in both $\sF$ and $\sG$.
Suppose that for some $x \in \cX$, we have $V \subseteq \sF|_x \cap \sG|_x$
satisfying one of the following conditions:
\begin{Ilist}
\itm $k=r$, and the subspace $V$ is not degenerate
for any nonzero linear combination $\langle,\dots,\rangle$ of the
$\langle,\dots,\rangle_i|_x$;
\itm $k =r+1$, $m=1$, and no $r$-dimensional subspace of $W \subseteq V$
is degenerate for $\langle,\dots,\rangle_1$;
\itm $r=3$, $k=5 \text{ or }6$, $m=1$, and no $2$-dimensional subspace 
$W \subseteq V$ is degenerate for $\langle,\dots,\rangle_1$.
\end{Ilist}
Then every component of $\cG(k,\sF\cap\sG)$ passing
through the point corresponding to $V$ has codimension at
most
$$k(2n-s-t)-m\binom{k}{r}$$
in $\cG(k,\sE)$.
\end{cor}

\section{Application to vector bundles on curves}

We consider the following situation.
Let $S$ be a scheme, $\cX$ an algebraic $S$-stack, and $\pi:\cC \to \cX$ a 
smooth, projective relative curve of genus $g$ over $\cX$. Let $\sL$ be a 
line bundle on $\cC$ of relative degree $d$, 
and $\sE$ a vector bundle of rank $r$, together with an isomorphism
$\psi:\det \sE \risom \sL$. 

We describe how to construct $r$-linear alternating forms on 
$\pi_* (\sE(D)/\sE(-(r-1)D))$.

\begin{prop}\label{prop:form-construct} In the above situation, suppose
we also have a morphism $\xi:\sL \to \Omega^1_{\cC/\cX}$, 
and $P_1,\dots,P_N:\cX \to \cC$ disjoint sections of $\pi$,
and set $D =\sum_i P_i$. Then we construct an alternating $r$-linear form 
$\langle ,\dots,\rangle_{\xi}$ on $\pi_* (\sE(D)/\sE(-(r-1)D))$ defined 
locally on $\cX$ by
$$\langle s_1,\dots,s_r\rangle=\sum_{i}^n 
\res_{P_i} (\xi \circ \psi)(\tilde{s}_1 \wedge \dots \wedge \tilde{s}_r),$$
where each $\tilde{s}_j$ is a representative in $\sE(D)$ of $s_j$ in a
neighborhood of $P_i$ (more precisely, in a neighborhood of the point
of $P_i$ lying over a given point of $\cX$). Moreover, this form is
compatible with base change.
\end{prop}

\begin{proof} The argument is largely the same the first half of the
proof of Lemma 5.1 of \cite{os16}. The main distinction is that we are
forced to use $\sE(-(r-1)D))$ as the appropriate generalization of
$\sE(-D)$, to ensure that if we take a wedge product with $r-1$ local
sections of $\sE(P_i)$, the result will still be regular at $P_i$, and
thus will have residue equal to $0$.
\end{proof}

Note that for $r>2$, the form constructed in Proposition 
\ref{prop:form-construct} is highly degenerate: in particular, the subbundle
$\pi_* (\sE(-D)/(\sE(-(r-1)D)))$ is always degenerate. Nonetheless, we see
that we can make these forms behave in a rather nondegenerate manner when
we restrict our attention to their values on global sections.

\begin{prop}\label{prop:subspace-injective} In the situation of Proposition
\ref{prop:form-construct}, given a field $K$ and
a $K$-valued point $x$ of $\cX$, and
$$V \subseteq \Gamma(\cC|_x, \sE|_x)$$
a $k$-dimensional space of global sections of $\sE|_x$, suppose that $\xi$
is not identically zero on the fiber over $x$, and that for some $n \leq k$,
we have that no $n$-dimensional subspace of $V$ is contained in a subbundle
of $\sE|_x$ of rank $r-2$. 

Then there exists some $N>0$ such that for all $N'>N$, and any
disjoint sections $P_1,\dots,P_{N'}$ of $\pi$, we have that 
the form $\langle ,\dots,\rangle_{\xi}$ has the property that no
subspace $W \subseteq V$ of dimension $n$ is degenerate for 
$\langle ,\dots,\rangle_{\xi}|_x$.
\end{prop}

Note that the statement of the proposition depends only on the fiber of
$\sE$ at $x$; the base stack $\cX$ plays no role.

\begin{proof} First, observe that the Grassmannanian of $n$-dimensional 
subspaces
of $V$ is of finite type, and the loci on which the subspaces have rank
at most $r-2$ is likewise of finite type. It then follows that there is
some $N''$ such that any $n$-dimensional subspace of $V$ can have rank
less than or equal to $r-2$ at at most $N''$ points of $\cC|_x$. Choose 
$N'>N''+2g-2-d$ (note that $2g-2-d\geq 0$ since 
$\pi_* \sHom(\sL,\Omega^1_{\cC/\cX}) \neq 0$).

It is clear that we have the decomposition
$$\pi_* (\sE(D)/\sE(-(r-1)D)) \cong 
\bigoplus _{i=1}^{N'} \pi_*(\sE(P_i)/\sE(-(r-1)P_i)).$$ 
By definition,
the form $\langle ,\dots,\rangle_{\xi}|_x$ is compatible with this
direct sum decomposition, so to show that a subspace is not degenerate,
it suffices to show that there exists some $i$ such that its image
in $\pi_*(\sE(P_i)/\sE(-(r-1)P_i))$ is not degenerate for the restriction
of $\langle ,\dots,\rangle_{\xi}|_x$ to $P_i$. Calculating locally at $P_i$, we 
see moreover that if $P_i$ is not a zero of the map $\xi$, then 
$\langle ,\dots,\rangle_{\xi}|_x$ induces an isomorphism
$$\bigwedge^{r-1} \pi_*(\sE/\sE(-P_i)) \risom \pi_* (\sE(P_i)/\sE)^*.$$
Let $W \subseteq V$ be an $n$-dimensional subspace; then by hypothesis
the restriction of $W$ to $P_i$ is contained in $\pi_*(\sE/\sE(-(r-1)P_i))$,
so by the above isomorphism, to show that $W$ is nondegenerate, it is
enough to see that the map 
$$\bigwedge^{r-1} W \to \bigwedge^{r-1} \pi_*(\sE/\sE(-P_i))$$
is nonzero,
or equivalently, that the sections comprising $W$ span a subspace of
dimension at least $r-1$ in the fiber of $\sE$ at $P_i$.

Now, we can have at most $2g-2-d$ points at which $\xi$ vanishes, and
at most $N''$ points at which the sections of $W$ span a subspace of
$\sE|_{P_i}$ having dimension strictly less than $r-1$, so by construction
there is necessarily some $i$ such that $\bigwedge^{r-1} W$ has nonzero
image in $\bigwedge^{r-1} \pi_*(\sE/\sE(-P_i))$, and we conclude that
$W$ is not degenerate for $\langle ,\dots,\rangle_{\xi}|_x$, as desired.
\end{proof}

We can now prove Theorem \ref{thm:main-1}. Indeed, we prove a more general
form of the theorem, allowing the determinant to vary in families. We
first recall the definition of the moduli stack in question:

\begin{defn}\label{def:mod-stack} Let $S$ be a scheme, and $C/S$ a 
smooth, projective relative curve. Given also, $d,r,k \in \ZZ$, with 
$r \geq 2$, $k \geq 1$, and $\sL \in \Pic^d(C)$, the stack 
$\cG^k_{r,\sL}(C/S)$
parametrizes triples $(\sE,\psi,V)$ over every $S$-scheme $T$, where $\sE$ 
is a vector bundle of rank $r$ on $C_T:= C \times_S T$, 
$\psi:\det \sE \risom \sL|_{C_T}$ is an isomorphism, and $V$ is a rank-$k$ 
subbundle of $p_{2*} \sE$, in the sense that it
is a locally free subsheaf such that for all $T' \to T$, we have that the
induced map $V|_{T'} \to p_{2*}' (\sE|_{C_T'})$ remains injective, where
$p_{2}':C_{T'} \to T'$ is the projection morphism.

In the case that $S$ is the spectrum of a field, we write simply
$\cG^k_{r,\sL}(C)$.
\end{defn}

The construction of $\cG^k_{r,\sL}(C/S)$ as a closed substack of a relative
Grassmannian over the moduli stack of vector bundles of rank $r$ and
determinant $\sL$ proceeds exactly as in the classical rank-$1$ case.

\begin{thm}\label{thm:main-families} Let $S$ be an equidimensional scheme
of finite type over a field,
 and $C$ be a smooth, projective relative curve over $S$ of
genus $g$. Suppose $\sL \in \Pic^d(C)$, and $h^1(C_s,\sL|_s)$ is constant
as $s \in S$ varies, and is at least $m$. Given 
$r \geq 2$, and $s \in S$, let $\sE$ be a vector bundle of rank $r$ on $C_s$ 
with determinant $\sL|_s$, and $V \subseteq H^0(C_s,\sE)$ a $k$-dimensional 
space of global 
sections. Suppose that in addition, one of the following conditions is
satisfied.
\begin{Ilist}
\itm $k=r$, and $V$ is not contained in any subbundle of $\sE$ of 
rank $r-2$.
\itm $k=r+1$, $m=1$, and no $r$-dimensional subspace of $V$ is contained 
in any subbundle of $\sE$ of rank $r-2$.
\itm $r=3$, $k=5 \text{ or }6$, $m=1$, and no $2$-dimensional subspace of 
$V$ is contained in any subbundle of $\sE$ of rank $1$.
\end{Ilist}

Then every component of $\cG^k_{r,\sL}(C/S)$ passing through the point 
corresponding to $(\sE,V)$ has dimension at least 
\begin{equation}
\dim S+\rho-g+m\binom{k}{r}.\end{equation}
\end{thm}

\begin{proof} First observe that since the statement is purely 
dimension-theoretic, we may assume that $S$ is reduced. Then, by Grauert's
theorem and Serre duality,
the pushforward of $\sHom(\sL,\Omega^1_{C/S})$
is locally free of rank at least $m$, and pushforward commutes with base
change. Since the statement is local on
$S$, we may suppose we have $m$ linearly independent sections of this
pushforward which remain linearly independent under base change, and
we use these together with Proposition \ref{prop:form-construct} to 
construct $m$ alternating forms.
Furthermore, since etale base change does not affect dimension, we may
assume we have disjoint sections $P_1,\dots,P_{N'}$ as in Proposition
\ref{prop:subspace-injective}. Using Corollary \ref{cor:iso-intersect}, the 
argument then proceeds almost identically to
the proof of Theorem 1.3 in \S 5 of \cite{os16}.
In the notation of \textit{loc.\ cit.}, the only difference
is that the sheaves $\sE_{\cU}(D')/\sE_{\cU}(-D')$ and 
$\sE_{\cU}/\sE_{\cU}(-D')$ should be replaced by 
$\sE_{\cU}(D')/\sE_{\cU}(-(r-1)D')$ and 
$\sE_{\cU}/\sE_{\cU}(-(r-1)D')$ respectively, and the resulting ranks and
dimension counts modified appropriately.
\end{proof}

Theorem \ref{thm:main-1} then follows as the special case for which the
base $S$ is a point.

\begin{rem} Note that the condition that $h^1(C,\sL)$ be constant in fibers
is an important one. Without it, not only does the argument fail, but the 
theorem fails as well. See Example \ref{ex:nonconstant-fail} below.
\end{rem}

\begin{rem} The case of varying but special determinants is important
when one wants to study components of the stack $\cG^k_{r,d}(C)$; see
for instance Example \ref{ex:farkas-ortega}, and Example 
\ref{ex:nonconstant-fail}. According to the theorem we 
may also let the curve vary in families, but this seems less important
outside the context of degeneration arguments. 
\end{rem}

\section{The case of rank $2$}

The basic strategy of our analysis in the case of rank $2$ is to carry
out dimension counts via a detailed analysis of the possibilities for
extensions of line bundles. By virtue of Theorem \ref{thm:main-1}, we
will only have to compute upper bounds on dimensions to get the desired
result. 

\begin{defn} Let $\cS_{\sL}$ denote the stack over $\cG^2_{2,\sL}(C)$
consisting of tuples $(\sE,\psi,V,s_1,s_2)$, where $(\sE,\psi,V)$ are as
in Definition \ref{def:mod-stack}, and $s_1,s_2$ are a basis of $V$.
Let $\cS_{\sL}^{\gg}$ denote the open substack obtained as the preimage of 
$\cG^{2,\gg}_{2,\sL}(C)$ in $\cS_{\sL}$. Given $d' \geq 0$, denote by
$\cS^{\gg}_{d',\sL}$ the locally closed substack of $\cS_{\sL}^{\gg}$ on
which $s_1$ vanishes along a divisor of degree $d'$. 
\end{defn}

Then $\cS_{\sL}$ is a $GL_2$-torsor over $\cG^2_{2,\sL}(C)$, and is
in particular smooth of relative dimension $4$ over $\cG^2_{2,\sL}(C)$.
As $d'$ varies, the stacks $\cS^{\gg}_{d',\sL}$ give a stratification of
$\cS^{\gg}_{\sL}$. If $D$ is the divisor of vanishing of $s_1$ for a
point of $\cS^{\gg}_{d',\sL}$, we see that $s_2$ gives a nonzero section
of $\sL(-D)$, so in particular we must have $d' \leq d$.

\begin{prop}\label{prop:exts1} Suppose $d \geq 0$, and $0 \leq d' \leq d$. 
Then 
$$\dim \cS^{\gg}_{d',\sL} \leq 2d+1-g+m-d'.$$

Moreover, $\cS^{\gg}_{0,\sL}$ is irreducible.
\end{prop}

\begin{proof} We have a morphism from
$\cS^{\gg}_{d',\sL}$ to the symmetric product $S^{d'} C$ by taking the
vanishing divisor of $s_1$, and we denote by
$\cS^{\gg}_{D,\sL}$ the fiber of this morphism over the point corresponding
to a given effective divisor $D$. For a given choice of $D$,
consider
the stack $\cE_{D}$ parametrizing pairs $(\eta,s)$, where $\eta$ is an
extension of $\sL(-D)$ by $\sO(D)$, and $s \in H^0(C,\sL(-D))$ is nonzero
and lifts inside $\eta$. We then have a morphism 
\begin{equation}\label{eq:strata-map}
\cS^{\gg}_{D,\sL} \to \cE_{D}
\end{equation}
obtained by letting $\eta$ be the extension induced by $s_1$, and
letting $s$ be the image of $s_2$ under the extension.

Write $\ell:=h^0(C,\sO(D))$. We see that the fibers of \eqref{eq:strata-map} 
are determined by the choice of $s_2$ lifting $s$, and thus the morphism 
has relative dimension $\ell$. 
Now, an extension of $\sL(-D)$ by 
$\sO(D)$ corresponds to an element of $H^1(\sL(-2D))$. The infinitesimal 
automorphisms of such an extension are in correspondence with 
$H^0(\sL^{-1}(2D))$,
so the dimension of the stack of extensions
is $-\chi(\sL^{-1}(2D))=d-2d'+g-1$.
Using
Serre duality, an extension corresponds to an element of
$H^0(C,\Omega^1_C \otimes \sL(-2D))^*$, which we still denote by $\eta$.
The section $s \in H^0(C,\sL(-D))$ lifts under $\eta$ if and only if 
the kernel of $\eta$ contains the image of the (injective) map 
$$\otimes s:H^0(C,\Omega^1_C(-D)) \to H^0(C,\Omega^1_C \otimes \sL(-2D)).$$
We have
$$h^0(C,\Omega^1_C(-D))=h^1(C,\sO(D))=\ell-d'-1+g,$$
so for a given $s \in H^0(C,\sL(-D))$, the dimension of
the choices for $\eta$ is 
$$d-2d'+g-1-(\ell-d'-1+g)=d-d'-\ell.$$
On the other hand, set 
$d''=d'-h^0(C,\sL)+h^0(C,\sL(-D))$.
There are $h^0(C,\sL(-D)) =d''-d'+d+1-g+m$ dimensions for $s$, so we 
conclude that $\cE_{D}$ has dimension $2d+1-g+m-2d'+d''-\ell$, and
$\cS^{\gg}_{D,\sL}$ has dimension $2d+1-g+m-2d'+d''$.
Finally, for a given value of $d''$, the corresponding stratum of $S^{d'} C$ 
has dimension at most $d'-d''$, so we find that
$\cS^{\gg}_{d',\sL}$ has dimension at most $2d+1-g+m-d'$, as desired.

Finally, to see that $\cS^{\gg}_{0,\sL}$ is irreducible, in the case $d'=0$
we necessarily have $D=0$ and our stratification is the trivial one
corresponding to $\ell=1$. We observe
that the space of choices for $s \in H^0(C,\sL)$ are irreducible, and
given a choice of $s$, the spaces of extensions $\eta$ is also irreducible.
Now, the preimage of $\cE_{0}$ is an open substack of
$\cS^{\gg}_{\sL}$, and hence we have a dimension lower bound as well,
concluding that $\cE_{0}$ is pure of dimension $2d-g+m$. It then
follows that every component of $\cE_{0}$ must dominate the
space of choices for $s$,
and thus that $\cE_{0}$ is irreducible.
By the same argument, we then conclude that $\cS^{\gg}_{0,\sL}$, being
smooth with connected fibers over $\cE_0$, is likewise irreducible. 
\end{proof}

\begin{defn} Let $\cG^{2,\dg}_{2,\sL}(C)$ be the closed substack of 
$\cG^{2}_{2,\sL}(C)$ on which $V$ is not generically generating. 
\end{defn}

\begin{prop}\label{prop:exts2} If $C$ is Brill-Noether general with respect 
to $\fg^{1}_{d'}$'s for all $d'>0$, we have
$$\dim \cG^{2,\dg}_{2,\sL}(C) \leq 2d-3-g+m,$$ 
and equality holds if and only if $h^0(C,\sL)=0$.
\end{prop}

\begin{proof} First observe that if $h^0(C,\sL)=0$, then
$\cG^{2,\dg}_{2,\sL}(C)= \cG^{2}_{2,\sL}(C)$, which has dimension at least
$2d-3-g+m$, so it is enough to show that 
$\dim \cG^{2,\dg}_{2,\sL}(C) \leq 2d-3-g-m$, with strict inequality when
$h^0(C,\sL)>0$.

Given $d'>0$, consider the stack $\cE'_{d'}$ parametrizing 
triples $(\sM,W,\eta)$, where $(\sM,W)$ is a $\fg^1_{d'}$ on $C$, and
$\eta$ is an extension of $\sL\otimes \sM^{-1}$ by $\sM$.
For a given $\sM$, the stack of extensions $\eta$ has dimension calculated
as before:
$$-\chi(\sL^{-1} \otimes \sM^2)=d-2d'+g-1.$$
Since $C$ is Brill-Noether general with
respect to $\fg^{1}_{d'}$'s, the dimension of the stack of pairs $(\sM,W)$ 
is $2d'-2-g-1$ (note that this is $1$ less than the classical number because
we have to take the stack dimension). We conclude that the dimension of
$\cE'_{d'}$ is $d-4$.

Now, noting that the form of the extension $\eta$ induces an isomorphism
$\det \sE \risom \sL$, we obtain a forgetful morphism 
\begin{equation}\label{eq:strata-map2}
\cE'_{d'} \to \cG^{2,\dg}_{2,\sL}(C).
\end{equation}
Moreover, any nontrivial automorphism of an object of $\cE'_{d'}$ induces a
nontrivial automorphism of $\sE$ and hence maps to a nontrivial automorphism
in $\cG^{2,\dg}_{2,\sL}(C)$, so we conclude that \eqref{eq:strata-map2} is
relatively representable by algebraic spaces, and in particularly has 
nonnegative relative dimension. 

Now, as $d'$ varies over all positive values, the union of the images
of the morphisms \eqref{eq:strata-map2} surjects onto 
$\cG^{2,\dg}_{2,\sL}(C)$, so we conclude that the dimension of 
$\cG^{2,\dg}_{2,\sL}(C)$ is at most the supremum of the dimensions of the
$\cE'_{d'}$, which is $d-4$.
Finally, 
$$2d-3-g+m-(d-4) = d+1-g+m = h^0(C,\sL),$$
so we obtain the desired statement.
\end{proof}

\begin{defn} Suppose $h^0(C,\sL)>0$ and $d>0$. Let $\cS^{\ns}_{\sL}$ be the 
closed substack of $\cS^{\gg}_{0,\sL}$
on which $\sE$ is not stable. Given $d' \geq \frac{d}{2}$, denote by
$\cS^{\ns}_{d',\sL}$ the locally closed substack of $\cS^{\ns}_{\sL}$ on
which $\sE$ has a destabilizing line subbundle of degree $d'$.
\end{defn}

Thus, as $d'$ varies, the $\cS^{\ns}_{d',\sL}$ give a stratification of
$\cS^{\ns}_{\sL}$.

\begin{prop}\label{prop:exts-stable} We have
$$\dim \cS^{\ns}_{d',\sL}< 2d+1-g+m$$
if $C$ is not hyperelliptic and $d>2$ or if $C$ is hyperelliptic and
$d>4$.
\end{prop}

\begin{proof} Consider the stack $\cE''_{d'}$ parametrizing tuples 
$(\eta,s,\sM,\iota)$, where $\eta$ is an extension of $\sL$ by $\sO$,
$s \in H^0(C,\sL)$ lifts in $\eta$, $\sM$ is a line bundle of degree $d'$,
and $\iota:\sM \to \sE$ imbeds $\sM$ as a line subbundle.
Given such a tuple, this yields
a map $\sM \to \sL$, which must be nonzero since $d'>0$. Thinking of
this map as a section $t \in H^0(C,\sL \otimes \sM^{-1})$, the condition
that it came from a map $\sM \to \sE$ is equivalent to the condition that
$t$ lifts in the extension $\eta \otimes \sM^{-1}$
$$0 \to \sM^{-1} \to \sE \otimes \sM^{-1} \to \sL \otimes \sM^{-1} \to 0.$$
As before, this in turn is equivalent to asking that the image of
$H^0(C,\Omega^1_C \otimes \sM)$ under the map
$$\otimes t: H^0(C,\Omega^1_C \otimes \sM) \to H^0(C,\Omega^1_C \otimes \sL)$$
be contained in the kernel of the
extension $\eta \otimes \sM^{-1}$, considered as an element of
$H^0(C,\Omega^1_C \otimes \sL)^*$.
We thus need to determine how this condition interacts with the condition
that $s$ must lift to $\sE$ as well.

Let $D =\dv s$, and $D' = \dv t$, so that $\deg D'=d-d'$. Let 
$D''=\gcd(D,D')$, and set $d''=\deg D''$.  Observe that the
condition that a nonzero element $s' \in H^0(C,\Omega^1_C \otimes \sL)$ be 
in the image of $H^0(C,\Omega^1_C \otimes \sM)$ under $\otimes t$ is
precisely equivalent to having $D' \leq \dv s'$, and similarly $s'$
is in the image of $H^0(C,\Omega^1_C)$ under $\otimes s$ if and only if
$D \leq \dv s'$. For a given $s$ and $t$, we want to compute the dimension
of the stack of extensions whose kernels contain both images, so we need to
compute the dimension of the span of the images. Since we know the 
dimensions of each image, it suffices to compute the dimension of the
intersection. We have that $s'$ is in the intersection of the images if
and only if $\lcm(D,D')=D+D'-D'' \leq \dv s'$, so the intersection of the
images is given by 
$$H^0(C,\Omega^1_C \otimes \sL(-D-D'+D'')) \cong H^0(C,\Omega^1_C(-D'+D'')).$$
First suppose that $\deg \Omega^1_C(-D'+D'') \geq 0$.
Then Clifford's theorem implies that
$$\frac{2g-2-d+d'+ d''}{2}+1=g-\frac{d-d'-d''}{2},$$
with equality possible only if $\Omega^1_C(-D'+D'')\cong \sO_C$, if
$D'=D''$, or if $C$ is hyperelliptic. Thus, the span of 
the images has dimension at least
$$g+(g-1+d')-(g-\frac{d-d'- d''}{2})=g-1+\frac{d+d'- d''}{2},$$
and the dimension of the choices of extensions for a given 
$s \in H^0(C,\sL)$ and $t \in H^0(C,\sL \otimes \sM^{-1})$ is at most
$$d+g-1-(g-1+\frac{d+d'-d''}{2})=\frac{d-d'+ d''}{2}.$$
As before, the choices for $s$ add $d+1-g+m$ dimensions, while choosing
the pair $(\sM,t)$ is just equivalent to choosing any effective divisor
of degree $d-d'$ containing $D''$, so adds $d-d'-d''$ dimensions. Since
$d' \geq \frac{d}{2}$, the total dimension of $\cE''_{d'}$ is at most 
$$\frac{5d-3d'- d''}{2}+1-g+m \leq \frac{7d}{4}+1-g+m.$$

Considering the stack of tuples $(\eta,s,\sM,\iota,s_2)$, where
$(\eta,s,\sM,\iota)$ is as in the definition of $\cE''_{d'}$, and $s_2$ is
a lift of $s$, we obtain a correspondence between
$\cS^{\ns}_{d',\sL}$ and $\cE''_{d'}$. A fiber of the map to
$\cE''_{d'}$ corresponds to the choices of $s_2$, which form a torsor over
$H^0(C,\sO)$. Thus, the fibers are $1$-dimensional. On the other hand, the
morphism to $\cS^{\ns}_{d',\sL}$ is surjective with fibers representable
by algebraic spaces, so we conclude that
$\dim \cS^{\ns}_{d',\sL} \leq \dim \cE''_{d'}+1$, yielding
$$\dim \cS^{\ns}_{d',\sL} \leq \frac{7d}{4}+2-g+m.$$

This already yields the desired inequality in the hyperelliptic case.
In the nonhyperelliptic case, if we had strict inequality in Clifford's
theorem, the two sides had to differ by at least $\frac{1}{2}$, so the
relevant open substack of $\cS^{\ns}_{d',\sL}$ has dimension bounded
by $\frac{7d+2}{4}+1-g+m$,
which is also enough to obtain the desired inequality. On the other hand,
if either $\Omega^1_C(-D'+D'')\cong \sO_C$ or $D'=D''$, we can calculate
directly, and in both cases obtain that the corresponding strata of
$\cS^{\ns}_{d',\sL}$ have dimension bounded by 
$$2d-d'+2-g+m \leq \frac{3d}{2}+2-g+m,$$
again yielding the desired inequality. Finally, if 
$\deg \Omega^1_C(-D'+D'') < 0$, we can again calculate directly, finding
$$\dim \cS^{\ns}_{d',\sL} \leq 2d+2-2g+m,$$
which gives the asserted inequality for $g \geq 2$.
\end{proof}

\begin{proof}[Proof of Theorem \ref{thm:main-2}] It is clear that if
$h^0(C,\sL)=0$, then $\cG^{2,\gg}_{2,\sL}(C)$ must be empty. Conversely,
if $h^0(C,\sL)>0$, then taking the trivial extension of $\sL$ by $\sO$
shows that $\cG^{2,\gg}_{2,\sL}(C)$ is nonempty. By Theorem \ref{thm:main-1},
it has dimension at least $\rho-g+m=2d-3-g+m$. Then, recalling that
$\cS^{\gg}_{d',\sL}$ is smooth of relative dimension $4$ over 
$\cG^{2,\gg}_{2,\sL}(C)$, Proposition \ref{prop:exts1} gives us the necessary
upper bound to conclude the dimension is equal to $\rho-g+m$.
This argument also shows that the open substack of pairs containing
a nowhere vanishing global section is dense in
$\cG^{2,\gg}_{2,\sL}(C)$, and in particular we conclude that
$\cG^{2,\gg}_{2,\sL}(C)$ is irreducible from the statement of Proposition
\ref{prop:exts1} that $\cS^{\gg}_{0,\sL}$ is irreducible.

Next, if $C$ is Brill-Noether general with respect to $\fg^1_{d'}$'s for 
all $d'$, then Proposition \ref{prop:exts2} implies that the complement of
$\cG^{2,\gg}_{2,\sL}(C)$ in $\cG^2_{2,\sL}(C)$ can have dimension at
most $\rho-g+m$, with equality if and only if $h^0(C,\sL)=0$, so we conclude 
that $\cG^{2}_{2,\sL}(C)$ is pure of dimension $\rho-g+m$, and that
$\cG^{2,\gg}_{2,\sL}(C)$ is dense in $\cG^2_{2,\sL}(C)$ when $h^0(C,\sL)>0$.

Finally, when $h^0(C,\sL)>0$, we see from Proposition \ref{prop:exts-stable}
that the complement of the stable locus in $\cG^{2,\gg}_{2,\sL}(C)$ must
have strictly smaller dimension under the stated hypotheses, so we conclude
that $\cG^{2,\st}_{2,\sL}(C)$ contains a nonempty open substack, as desired. 
\end{proof}

We also obtain a corollary for varying determinant loci in the case of rank 
$2$, which refines the main dimension result of Teixidor i Bigas in
\cite{te4}.

\begin{cor}\label{cor:rank-two-vary} Fix $d,g,m \geq 0$, let $C$ be a 
Brill-Noether general curve, and suppose that $\ell=d+1-g+m$ is nonnegative.
Let $S$ be the locally 
closed subvariety
$$W^{\ell-1}_d \smallsetminus W^{\ell}_d \subseteq \Pic^d(C),$$
and let $\sL$ be the restriction of the Poincare line bundle to $S \times C$. 
Then $\cG^2_{2,\sL}(C/S)$ is nonempty if and only if 
$m \ell \leq g$, and in this case it has pure dimension 
\begin{equation}\label{eq:rho-varying} 
\rho-(\ell-1) m.\end{equation}
Similarly, the open substack $\cG^{2,\gg}_{2,\sL}(C/S)$ is nonempty (necessarily
of the same dimension) if and only 
if $\ell > 0$ and $m \ell \leq g$.
\end{cor}

\begin{proof}
The classical Brill-Noether theorem gives that $S$ is non-empty and
$$\dim S = g-m\ell$$
if and only if $m \ell \leq g$. 
We then obtain the desired dimensional lower bound from Theorem 
\ref{thm:main-1}, and the corresponding upper bound from Theorem 
\ref{thm:main-2}, working fiber by fiber. The nonemptiness assertion
likewise follows from Theorem \ref{thm:main-2}. 
\end{proof}

\begin{rem} We observe from \eqref{eq:rho-varying} that there are two 
possibilities for getting dimension exactly $\rho$: either $\ell=1$, or 
$m=0$. For any given $d,g$, one of these is always possible. We also 
see that 
if we consider the degenerate locus, we can allow $\ell=0$ and $m>0$ and 
find that in this varying determinant situation, we actually obtain dimension
strictly greater than $\rho$ on the degenerate locus. This occurs if 
$m=g-d-1>0$, so $g > d+1$. In this case, we check that a degenerate pair 
must have an unstable underlying bundle, so this does not contradict
\cite{te4}.
\end{rem}

\section{Further discussion}

The arguments used to prove Theorem \ref{thm:main-1} show that for any
$k,r,m$, if $h^1(C,\sL)\geq m$, there is an open substack of
$\cG^k_{r,\sL}(C)$ satisfying the dimension lower bound of \eqref{eq:mod-rho}.
The difficulty is in translating the criterion of Proposition 
\ref{prop:iso-intersect} into a concrete nondegeneracy criterion describing 
this open substack, as for instance in the statement of Theorem 
\ref{thm:main-1}.  \textit{A priori}, we have no way of knowing even whether 
the open substack in question is ever nonempty. We observe that for $m \geq 3$
or $r \geq 3$, the formula of \eqref{eq:mod-rho} is in fact increasing in
$k$ for $k$ sufficiently large. This underlines the likelihood that 
nondegeneracy hypotheses will be necessary in these cases.

We now consider several examples, examining the necessity of nondegeneracy
hypotheses, and evaluating our predicted bounds in examples from the
literature of larger-dimensional Brill-Noether loci.

\begin{ex} Although the nondegeneracy hypothesis of Theorem \ref{thm:main-1}
is vacuous in rank $2$, and there are also no nondegeneracy hypotheses in
the main results of \cite{os16}, we mention that as soon as $m \geq 3$, 
even in rank $2$ we will need some nondegeneracy hypotheses in order for the
lower bound \eqref{eq:mod-rho} to be valid. Specifically, for any fixed 
$d$, choose $k$
very large with respect to $d$ and $g$, so that the only pairs $(\sE,V)$ with
$V \subseteq H^0(C,\sE)$, $\deg \sE=d$, and $\dim V=k$ must be degenerate,
with $V$ contained in some high-degree line subbundle of $\sE$.
Fix any line bundle
$\sL$ of degree $d$. Set $d'=k+g-1$, so that every line bundle $\sM$ of
degree $d'$ has $h^0(C,\sM)=k$, but no line bundle of smaller degree has
a $k$-dimensional space of global sections. Let 
$\cU_{d'} \subseteq \cG^k_{2,\sL}(C)$ be the open substack on which the
vector bundles have sublinebundles of degree at most $d'$. 

We then see that $\cU_{d'}$ is pure of dimension
$$2g-2-2d'+d=d-2k.$$
Indeed, by construction it consists entirely of bundles of
form $\sM \oplus (\sL \otimes \sM^{-1})$, with $\sM$ a line bundle of 
degree $d'$. Note that since $d'$ is large with respect to $d$ and $g$,
there are no nontrivial extensions of $\sL \otimes \sM^{-1}$ by $\sM$.
There is a $g$-dimensional space of choices for such a vector bundle,
and taking into account the fixed determinant condition,
the dimension of the automorphism group of each is
$$1+h^0(C,\sM^2 \otimes \sL^{-1})=2d'-d+2-g.$$
This gives the claimed formula for the dimension of $\cU_{d'}$.

In particular, we see that the dimension of $\cU_{d'}$ is decreasing in
$k$. On the other hand, we have already observed that for $m \geq 3$ and
$k$ large, our lower bound
$$\rho-g+m\binom{k}{2}$$
is increasing in $k$. We thus see that whatever nondegeneracy condition
is required for this lower bound to hold, it must be violated by the
present examples.
\end{ex}

In the next two examples, we see that in two interesting examples of
Brill-Noether loci of larger than expected dimension, the discrepancy
is explained by our techniques, and that moreover our lower bound is
sharp in these cases.

\begin{ex}\label{ex:farkas-ortega} In \cite{f-o1},
Farkas and Ortega study the case of odd genus $2a+1$, with 
rank $2$, degree $2a+4$ and $k=4$. They find that while in this case 
$\rho=1$, the dimension of the coarse moduli space of 
$\cG^{4,\st}_{2,2a+4}(C)$ is $2$.
As far as we are aware, this is the only known example
of larger-than-expected dimension in rank $2$ other than those described
explicitly by special determinants. We note however that this example
is nonetheless explained by our work in \cite{os16}: indeed, their 
analysis shows that
$\cG^{4,\st}_{2,2a+4}(C)$ is supported entirely over the locus of 
determinants $\sL$ having $h^1(C,\sL) > 0$. This locus has dimension
$g-5$, and for a fixed such $\sL$ we know that the dimension of 
$\cG^{4}_{2,\sL}(C)$ is at least 
$$\rho-g+\binom{k}{2}=\rho-g+6,$$
so if we allow $\rho$ to vary we conclude that the dimension of
$\cG^{4,\st}_{2,2a+4}(C)$ should be at least $\rho+1=2$, as oberved by 
Farkas and Ortega.
\end{ex}

\begin{ex}\label{ex:3-6} As discussed in \cite{l-m-n1},
Mukai has shown that for a general curve of genus
$9$, there exists a unique stable vector bundle of rank $3$ and degree $16$ 
with a $6$-dimensional space of global sections. In this case, $\rho=-11$. On
the other hand, this vector bundle has canonical determinant, and the modified
expected dimension arising from Theorem \ref{thm:main-1} is
$$\rho-g+\binom{k}{r}=-11-9+\binom{6}{3}=0.$$
In addition, one checks using the generality of the curve that in this case,
stability of the vector bundle implies the nondegeneracy hypothesis of the 
theorem. 
Thus, we see that in at least one interesting example, not only does 
\eqref{eq:mod-rho} give a valid lower bound for the dimension, but it is 
in fact sharp, at least on the stable locus.

Also in \cite{l-m-n1}, Lange, Mercat and Newstead show that on a general 
curve of genus $11$, there exist stable bundles of rank $3$ and degree $20$
with a $6$-dimensional space of global sections, although in this
case $\rho=-5$. These bundles also have canonical determinant, so we
again find that our modified expected dimension is nonnegative, in this
case equal to $4$.
\end{ex}

Finally, we discuss the necessity of restricting to determinant loci
with constant $h^1$ in Theorem \ref{thm:main-families}.

\begin{ex}\label{ex:nonconstant-fail}
In Example 6.1 of \cite{os16},
we consider the case $r=2$, $k=2$, and $d=g-2$.
If we let $S$ be all of $\Pic^d(C)$, and $\sL$ the Poincare line
bundle, then every fiber has nonzero $h^1$, so if Theorem 
\ref{thm:main-families} remained valid without the hypothesis that 
$h^1$ is constant, we could use the $m=1$ case to conclude that
if the relative stack $\cG^2_{2,g-2}(C/S)$ is nonempty, every component has 
dimension at least $\rho+1$. However, the stable locus of 
$\cG^2_{2,g-2}(C/S)$ is in fact nonempty of dimension $\rho$. The explanation 
is that this stable locus is supported over the locus of $\Pic^d(C)$ on 
which fibers of $\sL$ have $h^1$ at least $2$.

Note however that $\cG^2_{2,g-2}(C/S)$ is in fact nonempty over all of
$\Pic^d(C)$ on the degenerate locus, so if we let $S$ be the locus of
$\Pic^d(C)$ on which fibers of $\sL$ have $h^1$ exactly equal to $1$, 
the theorem does imply that $\cG^2_{2,g-2}(C/S)$ has dimension at least
$\rho+1$. We conclude that while a nondegeneracy hypothesis is not
necessary for Theorem \ref{thm:main-2}, it is necessary in the varying
determinant case treated by Teixidor i Bigas.
\end{ex}

Of course, the ultimate goal of the program is to produce modified expected
dimensions which are actually sharp. Theorem \ref{thm:main-2} together
with the more general situation for cases (I) and (II) of Theorem 
\ref{thm:main-1} discussed in the introduction provide some simple cases
where our bounds are already sharp, but are undoubtedly extremely special.
It seems likely that there is a 
degree of inductive structure to the problem, and thus that it makes sense
to focus attention initially on rank $2$. In light of Theorem 1.1 of
\cite{os16}, it is evident that even if we prove dimension bounds as
discussed above in full generality, we will not have sharp results. It
is natural to speculate that given a determinant line bundle $\sL$, there
should be a sequence of expected dimensions, associated to the sequence 
$\delta_1,\delta_2,\dots$, where $\delta_m$ is the minimal degree of an
effective divisor $D_m$ such that $h^1(C,\sL(-D_m)) \geq m$. It is then
possible that the correct expected dimension would be furnished by the
maximum value of this sequence. While it seems likely that Theorem 1.1
of \cite{os16} gives the correct value for $\delta_1$, the analysis for
$\delta_i$ with $i>1$ is subtler, and we do not hazard a guess as to the
correct value.

\bibliographystyle{hamsplain}
\bibliography{hgen}

\end{document}